\documentclass{article}
\usepackage[utf8]{inputenc}
\usepackage{fullpage}
\usepackage{mathtools}
\usepackage{amsfonts,amsmath,amssymb,amsthm}

\def\lf{\left\lfloor}   
\def\rf{\right\rfloor}

\def\snd{\operatorname{snd}}

\newtheorem{theorem}{Theorem}[section]

\newtheorem{example}[theorem]{Example}

\newtheorem{corollary}[theorem]{Corollary}

\title{On small $n$-uniform hypergraphs with positive discrepancy}
\author{Danila Cherkashin\footnote{Chebyshev Laboratory, St. Petersburg State University, 14th Line V.O., 29B, Saint Petersburg 199178 Russia; 
Moscow Institute of Physics and Technology, Lab of advanced combinatorics and network applications, Institutsky lane 9, Dolgoprudny, Moscow region, 141700, Russia;  National Research University Higher School of Economics, Soyuza Pechatnikov str., 16, St. Petersburg, Russian Federation.}, 
Fedor Petrov\footnote{Saint Petersburg State University, Faculty of Mathematics and Mechanics; St.~Petersburg Department of V.~A.~Steklov Institute of Mathematics of the Russian Academy of Sciences.}}

\date{March 2017}

\begin{document}

\maketitle

\begin{abstract}
A two-coloring of the vertices $V$ of the hypergraph $H=(V, E)$ by red and blue has discrepancy
$d$ if $d$ is the largest difference between the number of red and blue points in any edge.
Let $f(n)$ be the fewest number of edges in an $n$-uniform hypergraph without a coloring with discrepancy $0$. 
Erd{\H o}s and S{\' o}s asked: is $f(n)$ unbounded?

N. Alon, D. J. Kleitman, C. Pomerance, M. Saks and P. Seymour~\cite{AKPSS} proved upper and lower 
bounds in terms of the smallest non-divisor ($\snd$) of $n$ (see (\ref{1})). 
We refine the upper bound as follows:
$$f (n) \leq c \log \snd {n}.$$
\end{abstract}

\vskip+0.2cm


\section{Introduction}

A hypergraph is a pair $(V, E)$, where $V$ is a finite set whose elements are called vertices and $E$ is a family of subsets of $V$, called edges. 
A hypergraph is $n$-uniform if every edge has size $n$. 
A vertex $2$-coloring of a hypergraph $(V , E)$ is a map $\pi : V \rightarrow \{1, 2\}$.

The \textit{discrepancy} of a coloring is the maximum 
over all edges
of the difference between the number of vertices of 
two colors in the edge. 
The \textit{discrepancy} of a hypergraph is the minimum 
discrepancy of a coloring of this hypergraph.
The general discrepancy theory is set out in~\cite{alon2016probabilistic,matouvsek1999geometric,chen2014panorama}.

Let $f(n)$ be the minimal number of edges in an $n$-uniform hypergraph (all edges have size $n$) having positive discrepancy.
Obviously, if $2 \nmid n$ then $f(n) = 1$; if $2|n$ but $4 \nmid n$ then $f(n) = 3$.
Erd\H{o}s and S\H{o}s asked whether $f(n)$ is
bounded or not. N. Alon, D. J. Kleitman, C. Pomerance, M. Saks and P. Seymour~\cite{AKPSS} proved the following Theorem, showing in particular that $f(n)$ is unbounded.

\begin{theorem}
Let $n$ be an integer such that $4\,|\,n$. Then
\begin{equation}
c_1 \frac{\log \snd (n/2)}{\log \log \snd (n/2)} \leq f(n) \leq c_2 \frac{\log^3 \snd (n/2)}{\log \log \snd (n/2)},
\label{1}
\end{equation}
where $\snd (x)$ stands for the least positive integer that does not divide $x$.
\end{theorem}

\noindent To prove the upper bound they introduced several quantities.
Let $\mathcal{M}$ denote the set of all matrices $M$ with entries in $\{0, 1\}$
such that the equation $Mx = e$ has exactly one non-negative solution (here $e$ stands for the vector with all entries equal to $1$). 
This unique solution is denoted $x^M$. Let $z(M)$ be the least integer such that $z(M)x^M$ is integer and let $y^M=z(M)x^M$. 
For each positive integer $n$, let $t(n)$ be the least $r$ such that there exists a matrix $M \in \mathcal{M}$ with $r$
rows such that $z(M)=n$ (obviously, $t (n) \leq n+1$ because $z(J_{n+1}-I_{n+1}) = n$, where $J_{n+1}$ is the $(n+1) \times (n+1)$ matrix with unit entries; $I_{n+1}$ is the $(n+1) \times (n+1)$ identity matrix). 
The upper bound in~(\ref{1}) follows from the inequality $f (n) \leq t(m)$ for such $m$ that $\lf \frac{n}{m} \rf$ is odd.

Then N. Alon and V. H. V{\~u}~\cite{AV} showed that $t(m) \leq (2+o(1)) \frac{\log m}{\log \log m}$ for infinitely many $m$.
However they marked that trueness of inequality $t(m) \leq c \log m$ for arbitrary $m$ is not clear.

Our main result is the following

\begin{theorem}
Let $n$ be a positive integer number. Then 
\begin{equation}
f (n) \leq c \log{ \snd (n)}.
\end{equation}
for some constant $c > 0$.
\label{main}
\end{theorem}

\begin{corollary}
Let $n$ be a positive integer number. Then 
$$
f (n) \leq c \log \log {n}.
$$
for some constant $c > 0$.
\end{corollary}

The construction of the hypergraph
with positive discrepancy which yields  Theorem \ref{main} 
uses a matrix with determinant $\snd(n)$ and small entries satisfying some additional technical properties. Before coming to a general construction
we give an example with a specific $2\times 2$ matrix which shows
the vague idea.

\section{Example}

\begin{example}
\label{example}
Let us consider the matrix $A = \begin{pmatrix}
3 & 5 \\
1 & 8 
\end{pmatrix}$ and suppose that $n$ is not divisible on $\det A = 19$.
Consider the system 
\begin{equation}
A\begin{pmatrix}
a \\
b
\end{pmatrix} = \begin{pmatrix}
n \\
n + t
\end{pmatrix}.
\label{ex1}
\end{equation}

The solution of the system is $a = (3n-5t)/19$, $b = (2n+3t)/19$, which is integral if and only if $t = 12n$ (mod 19) i. e. $t$ has prescribed residue modulo 19.
Since $n$ is not divisible on $19$, $t$ is not equal to zero modulo $19$.
So one can choose $-19 < t < 19$ such that $t$ has prescribed residue modulo 19 and $t$ is odd. 
Also, assume that $n/8 > t > -2n/3$ which is certainly true if $n > 200$. Then $a$ and $b$ are positive and also $b > t$ and $a$, $b$ tend to infinity simultaneously with $n$.

Let us construct an $n$-uniform hypergraph $H$ with positive discrepancy.
Consider disjoint vertex sets $A_1, A_2, A_3$ of size $a$ and $B_1, \dots, B_8$ of size $b$.
If $t < 0$ then consider a vertex set $T$ of size $|t|$ and set $C := B_1 \cup T$; if $t > 0$ let $T$ be a $t$-vertex subset of $B_1$ and define $C := B_1 \setminus T$.
The edges of $H$ are listed:
$$A_1 \cup A_2 \cup A_3 \cup B_1 \cup B_2 \cup B_3 \cup B_4 \cup B_5$$
$$A_1 \cup A_2 \cup A_3 \cup B_1 \cup B_2 \cup B_3 \cup B_4 \cup B_6$$
$$A_1 \cup A_2 \cup A_3 \cup B_1 \cup B_2 \cup B_3 \cup B_4 \cup B_7$$
$$A_1 \cup A_2 \cup A_3 \cup B_1 \cup B_2 \cup B_3 \cup B_4 \cup B_8$$
$$A_1 \cup A_2 \cup A_3 \cup B_2 \cup B_3 \cup B_4 \cup B_5 \cup B_8$$
$$A_1 \cup A_2 \cup A_3 \cup B_1 \cup B_3 \cup B_4 \cup B_5 \cup B_8$$
$$A_1 \cup A_2 \cup A_3 \cup B_1 \cup B_2 \cup B_4 \cup B_5 \cup B_8$$
$$A_1 \cup A_2 \cup A_3 \cup B_1 \cup B_2 \cup B_3 \cup B_5 \cup B_8$$
$$A_1 \cup C \cup B_2 \cup B_3 \cup B_4 \cup B_5 \cup B_6 \cup B_7 \cup B_8$$
$$A_2 \cup C \cup B_2 \cup B_3 \cup B_4 \cup B_5 \cup B_6 \cup B_7 \cup B_8$$
$$A_3 \cup C \cup B_2 \cup B_3 \cup B_4 \cup B_5 \cup B_6 \cup B_7 \cup B_8.$$

Obviously, if $H$ has a coloring with discrepancy $0$, then $d(B_5) = d(B_6)$, where $d(X)$ is the difference between blue and red vertices in $X$, because the second edge can be reached by replacing $B_5$ on $B_6$ in the first edge. Similarly one can deduce that $d(A_i) = d(A_j)$ and $d(B_i) = d(B_j)$ for all pairs $i$, $j$.
So one can put $k := d(A_i)$, $l := d(B_i)$. Because of the first edge we have $3k + 5l = 0$. Obviously, $k$ and $l$ are odd numbers, so the minimal solution is $k = 5$, $l = -3$ (or $k = -5$, $l = 3$ which is the same because of red-blue symmetry). But then the last edge gives $|k + 8l| \leq |t|$ which contradicts with $|k + 8l| \geq 19 > |t|$.

So we got an example if $19\, \nmid \,n$ and $n > 200$ of an $n$-uniform hypergraph with $11$ edges and positive discrepancy.
\end{example}

The number of edges in this example equals $11=3+8$, the sum of maximal
entries in the columns of $A$. This is essentially (up to multiplicative
constant) the general property of our construction.

\section{Proofs}

\begin{proof}[Proof of Theorem~\ref{main}]
Let us denote $\snd(n)$ by $q$. We should construct a hypergraph with at most $c \log q$
edges and positive discrepancy.
Take $m$ such that $2^{m}-1 \leq q\leq 2^{m+1}-2$. Then 
$$q-(2^m-1)=\sum_{i=0}^{m-1} \varepsilon_i 2^i \ \ \mbox{for some}\ \ \varepsilon_i\in \{0,1\},$$
therefore 
$$q = \sum_{i=0}^{m-1} \eta_i 2^i, \ \ \mbox{where} \ \ \eta_i=1+\varepsilon_i\in \{1,2\}.$$

\noindent Consider $m$ vectors in $\mathbb{Z}^m$: 
$$v_0=(\eta_0,\dots,\eta_{m-1}),$$
$$v_i=(\eta_0,\dots,\eta_{i-2},\eta_{i-1}+2,\eta_{i}-1,\eta_{i+1},\dots,\eta_{m-1}) \ \ \mbox{for} \ \ i=1,\dots,m-1, \mbox{i. e.}$$
$$v_{i,k} = \begin{cases} \eta_k, \ \ k \neq i, i-1 \\ \eta_k-1, \ \ k = i \\ \eta_k + 2, \ \ k = i-1.\end{cases}$$
Note that the vector $u=(1,2,\dots,2^{m-1})$ satisfies a system of linear equations 
$$\langle v_i,u\rangle=q; \ \ i=0,\dots,m-1.$$ 
Assume that $q$ is odd. Choose odd $\delta \in (-q,q)$ such that 
$x_0 := \frac{n+\eta_{m-1}\delta}{q}$ is integer. Define 
$$x_i := 2^i x_0 \ \ \mbox{for} \ \  i=1,\dots,m-2; \ \ \ x_{m-1} := 2^{m-1}x_0-\delta,$$ then the vector
$x=(x_0,\dots,x_{m-1})$ satisfies $\langle v_i,x\rangle=n$ for $i=0,\dots,m-2$, $\langle v_{m-1},x\rangle=n+\delta$.

In the case $q = 2^m \geq 8$ we have $n \equiv 2^{m-1} \pmod q$ and $\eta_0=2,\eta_1=\dots=\eta_{m-1}=1$.

Choose $x=(x_0,\dots,x_{m-1})$ so that $\langle v_1,x\rangle=\langle v_{m-1},x\rangle=n+1$ 
and $\langle v_i,x\rangle=n$ for $i=0,2,3,\dots,m-2$. The solution is given by 
$$x_0 := \frac{n+2^{m-1}}{q}; \ \  x_1 := 2x_0-1;  \ \ x_i := 2^{i-1}x_1 \ \ \mbox{for} \ \  i=2,\dots,m-2; \ \ x_{m-1} :=2^{m-2}x_1-1.$$

Now let us construct a hypergraph in the following way: for $i = 0, \dots m-1$ let us take $4$ sets $A_i^j$ ($j = 1, \dots, 4$) of vertices of size $x_i$ such that all 
$4m$ sets $A_i^j$ are disjoint.
Let the edge $e_0$ be the union of $A_i^j$ over $0 \leq i \leq m-1$ and $1 \leq j \leq \eta_i$.
By the choice of $x_i$ and $\eta_i$ we have $|e_0| = n$.
Then we add an edge 
$$\bigcup_{0 \leq i \leq m-1} \bigcup_{\substack{1 \leq j \leq \eta_i \ \ \mbox{for} \ \ i \neq k \\ j \in R \ \ \mbox{for} \ \ i = k}} A_i^j$$
for every $k$ and for every $R \subset [4]$ such that $|R| = \eta_k$. 
Clearly there are at most $6m$ such edges. Let us say that they form \textit{the first collection of edges}. 
Finally, for every $1 \leq k \leq m-1$ we add the edge 
$$\bigcup_{0 \leq i \leq m-1} \bigcup_{\substack{1 \leq j \leq \eta_i \ \ \mbox{for} \ \ i \neq k, k-1 \\ 1 \leq j \leq \eta_i+2 \ \ \mbox{for} \ \ i = k-1 \\ 1 \leq j \leq \eta_i-1 \ \ \mbox{for} \ \ i = k}} A_i^j,$$ which form \textit{the second collection of edges}.

Summing up we have hypergraph with at most $7m$ edges; at most $2$ of them have size not equal to $n$.
Let us correct these edges in the simplest way: if an edge has size less than $n$ then we add arbitrary vertices;
if an edge has size greater than $n$ then we exclude arbitrary vertices. 

Suppose that our hypergraph has discrepancy $0$, so it has a proper coloring $\pi$. 
For every set $A_i^j$ denote by $d(A_i^j)$ the difference between the numbers of red and blue vertices of $\pi$ in $A_i^j$.
Obviously, $d(A_i^{j_1}) = d(A_i^{j_2})$ because there are edges $e_1$, $e_2$ from the first collection such that $e_2$ can be obtained from $e_1$ by the replacement of $A_i^{j_1}$ to $A_i^{j_2}$.
So we may write $d_i$ instead of $d(A_i^j)$.

If $q$ is odd then the vector $d=(d_0,\dots,d_{m-1})$ satisfies 
$$\langle v_i,d\rangle = 0 \ \ \mbox{for} \ \ i = 0,1,\dots,m-2 \ \ \mbox{and} \ \ \langle v_{m-1},d\rangle=s$$
for some odd $s\in (-q,q)$. 
Considering consequent differences of this equations we get 
$$d_i=2^id_0 \ \ \mbox{for} \ \ i=0,\dots,m-2; \ \ d_{m-1}=2^{m-1}d_0-s; \ \ 0=\sum \eta_i d_i=d_0q-\eta_{m-1}s,$$
but this fails modulo $q$.
A contradiction. In the case $q=2^m$ we get a similar contradiction, as $(2^{m-1}-1)\pm 1$ is not divisible by $2^m$.

Thus we get a hypergraph on at most $7m = O(\log q)$ edges with positive discrepancy, the claim is proven.
\end{proof}

\section{Discussion}

\begin{itemize}

\item  In fact, during the proof we have constructed a matrix of size of $O(\log k)$ with bounded integer coefficients and with determinant $k := \snd (n)$. 
By Hadamard inequality, the determinant $k$ of $m\times m$
matrix with bounded coefficients satisfies 
$k = O(\sqrt{m})^m$, thus $\log k=O(m\log m)$, $m\geq \text{const}\,\cdot \log k/\log \log k$. 
We suppose that actually a matrix of size 
$O(\log k/\log \log k)$ with bounded
integer coefficients and determinant $k$
always exists; and moreover, it may 
be chosen satisfying additional properties 
which allow to replace the main estimate 
with $f(n)\leq c\log \snd (n)/\log\log \snd(n)$ (which asymptotically coincides with the lower bound).

\item It turns out, that for a fixed value of $q = \snd (n)$ and some values of $n$ modulo $q$, a hypergraph, constructions of above type have the discrepancy separated from zero. In particular, in Example~\ref{example} the choice $n \in \{\pm 4, \pm 7\}$ modulo 19 leads to the discrepancy 6.

\item For fixed $r$ and large enough $n$ using Theorem~\ref{main} one can construct an $n$-uniform hypergraph with discrepancy at least $r$ and $O(\ln \snd{[n/r]})^r$ edges (here $[x]$ stands for the nearest integer to $x$), as follows: 
let $H_0$ be a hypergraph realizing $f([n/r])$, $H_1, \dots, H_{2r-1}$ be vertex-disjoint copies of $H$.
Let $V := V(H_1) \sqcup \dots \sqcup V(H_{2r-1})$, $E := \{\sqcup e_{i}\ | \ i \in A \subset [2r-1], |A| = r\}$.
By the construction, every $H_i$ has discrepancy at least $2$; so by pigeonhole principle $(V, E)$ has discrepancy at least $2r$.
Define $l := r[n/r] - n$.
Finally, if $l > 0$, then exclude arbitrary $l$ vertices from every edge $e \in E$; else add arbitrary $l$ vertices to every edge $e \in E$; denote the result by $H$.
By definition $l \leq r$, so the discrepancy of $H$ is at least $r$.
Since $|E(H_i)| = f([n/r])$, we have 
$$|E(H)| = \binom {2r-1}{r} f([n/r])^r = O(\ln \snd{[n/r]})^r \leq O(\ln \ln n)^{r}.$$

\item A.~Raigorodskii independently asked the same question in a more general form: 
he introduced the quantity $m_k(n)$ that is the minimal number of edges in a hypergraph without a vertex 2-coloring such that every edge has at least $k$ blue vertices and at least $k$ red vertices. So $m_k(n)$ is the minimal number of a edges in a hypergraph with discrepancy at least $n - 2k + 2$, in particular $f(n) = m_{n/2} (n)$ for even $n$.

For the history and the best known bounds on $m_k(n)$ see~\cite{teplyakov2013upperENG}. Note that our result replaces the bound 
$m_k(2k + r) = O(\ln k)^{r+1}$~\cite{cherkashin2011twoENG} with $m_k(2k + r) = O(\ln \ln k)^{r+1}$ for a constant $r$.
It worth noting, that in the case $n = O(k)$ the behavior of $m_k(n)$ is completely unclear.

\end{itemize}

\noindent\textbf{Acknowledgements.} 
The work was supported by the Russian Scientific Foundation grant 16-11-10014. The authors are grateful to 
A.~Raigorodskii for the introduction to the problem, to N.~Alon for directing our attention to the paper \cite{AKPSS}
and fruitful discussions and to N. Rastegaev for a very careful reading of the draft of the paper. 

\bibliographystyle{plain}
\bibliography{main}

\end{document}